\documentclass[letterpaper, 10 pt, journal, twoside]{amsart}
\usepackage[utf8]{inputenc}
\usepackage{mathptmx} % assumes new font selection scheme installed
\usepackage{newtxtext,newtxmath}

\usepackage{enumitem}

\usepackage{amsthm}
\usepackage{dsfont,url}
\usepackage{mathtools}
\usepackage[usenames, dvipsnames]{xcolor}
\usepackage[%
	pdftitle={Constrained stochastic predictive control under output feedback},%
	pdfauthor={Prabhat K. Mishra, Debasish Chatterjee, Daniel E Quevedo},%
	pdfcreator={Debasish Chatterjee},%
	colorlinks={true},%
	citecolor={Blue},%
	linkcolor={Sepia}%
]{hyperref}
\usepackage{float}
\usepackage{newfloat}
\usepackage{caption}
\DeclareFloatingEnvironment[fileext=cmh,placement={!ht},name=List]{myfloat}
\usepackage{amsmath} % assumes amsmath package installed
\usepackage{amssymb}  % assumes amsmath package installed
\usepackage{pgf,tikz}
\usepackage{adjustbox}
\usepackage{verbatim} 
\usepackage{xargs}
\usepackage{cite}
\usetikzlibrary{plotmarks}
\usetikzlibrary{chains}
\usepackage{pgfplots}
\pgfplotsset{compat=newest}
\pgfplotsset{plot coordinates/math parser=false}
\newlength\figureheight
\newlength\figurewidth
\usetikzlibrary{shapes,arrows}
\usetikzlibrary{positioning}
\usetikzlibrary{decorations,decorations.pathmorphing,decorations.pathreplacing}
\usetikzlibrary{shapes.misc}
\pgfdeclaredecoration{switch}{initial}
{
    \state{initial}[width=+\pgfdecoratedinputsegmentremainingdistance]
    {
        \pgfsetlinewidth{1pt}
        \pgfpathcircle{\pgfpoint{0}{0}}{0.05\pgfdecoratedinputsegmentremainingdistance}
        \pgfpathlineto{\pgfpoint{0.25\pgfdecoratedinputsegmentremainingdistance}{0}}
        \pgfpathcircle{\pgfpoint{0.25\pgfdecoratedinputsegmentremainingdistance}{0}}%
            {0.05\pgfdecoratedinputsegmentremainingdistance}
        \pgfpathlineto{\pgfpoint{0.75\pgfdecoratedinputsegmentremainingdistance}{\pgfdecorationsegmentlength}}
        \pgfpathcircle{\pgfpoint{0.75\pgfdecoratedinputsegmentremainingdistance}{0}}%
            {0.05\pgfdecoratedinputsegmentremainingdistance}
        \pgfpathlineto{\pgfpoint{\pgfdecoratedinputsegmentremainingdistance}{0}}
        \pgfpathcircle{\pgfpoint{\pgfdecoratedinputsegmentremainingdistance}{0}}%
            {0.05\pgfdecoratedinputsegmentremainingdistance}
    }
    \state{final}
    {
        \pgfpathmoveto{\pgfpointdecoratedpathlast}
    }
}
\pgfmathdeclarefunction{gauss}{2}{%
  \pgfmathparse{1/(#2*sqrt(2*pi))*exp(-((x-#1)^2)/(2*#2^2))}%
}
\DeclareMathOperator*{\minimize}{min}
\DeclareMathOperator*{\sbjto}{s.\ t.}

\DeclareMathOperator{\sat}{sat}
\DeclareMathOperator{\trace}{tr}

\renewcommand{\le}{\leqslant}
\renewcommand{\leq}{\leqslant}

\renewcommand{\geq}{\geqslant}

\newcommand{\R}{\mathds{R}}
\newcommand{\Nz}{\mathds{N}}
\newcommand{\N}{\mathds{N}^\ast}
\newcommand{\EE}{\mathds{E}}
\newcommand{\PP}{\mathds{P}}

\newcommand{\bmat}[1]{\begin{bmatrix}#1\end{bmatrix}}
\newcommand{\abs}[1]{\left|#1\right|}
\newcommand{\norm}[1]{\left\|#1\right\|}
\newcommand{\secref}[1]{\S\ref{#1}}

\newcommand{\transp}{^\top}

\newcommand{\inverse}{^{-1}}

\newcommand{\zeros}{\mathbf{0}}

\newcommand{\st}{x}

\newcommand{\dortho}{d_o}
\newcommand{\dschur}{d_s}

\newcommand{\A}{A}
\newcommand{\calA}{\mathcal{A}}
\newcommand{\Aortho}{\A_o}
\newcommand{\Aschur}{\A_s}
\newcommand{\B}{B}
\newcommand{\calB}{\mathcal{B}}
\newcommand{\calD}{\mathcal{D}}
\newcommand{\calC}{\mathcal{C}}
\newcommand{\calF}{\mathcal{F}}
\newcommand{\calG}{\mathcal{G}}
\newcommand{\calH}{\mathcal{H}}
\newcommand{\Bortho}{\B_o}
\newcommand{\Bschur}{\B_s}
\newcommand{\meas}{y}
\newcommand{\C}{C}
\newcommand{\control}{u}
\newcommand{\controlset}{\mathds{U}}
\newcommand{\wnoise}{w}

\newcommand{\mnoise}{\varsigma}
\newcommand{\calQ}{\mathcal{Q}}
\newcommand{\calR}{\mathcal{R}}

\newcommand{\sigalg}{\mathfrak{F}}

\newcommand{\YY}{\mathfrak Y}

\newcommand{\lra}{\longrightarrow}

\newcommand{\ee}{\mathfrak{\psi}}
\newcommand{\offset}{\boldsymbol{\eta}}
\newcommand{\gain}{\boldsymbol{\Theta}}
\newcommand{\authority}{u_{\max}}
\newcommand{\reachindex}{\kappa}
\newcommand{\reachab}{\mathrm{R}}
\newcommand{\Let}{\coloneqq}
\newcommand{\teL}{\eqqcolon}

\newtheorem{definition}{\sc Definition}
\newtheorem{example}{\sc Example}

\newtheorem{theorem}{\sc Theorem}
\theoremstyle{theorem}
\newtheorem{lemma}{\sc Lemma}
\newtheorem{proposition}{\sc Proposition}
\allowdisplaybreaks
\title{Output feedback stable stochastic predictive control with hard control constraints }
\author[P.\ K.\ Mishra]{Prabhat Kumar Mishra}
\author[D.\ Chatterjee]{Debasish Chatterjee}
\author[D.\ E.\ Quevedo]{Daniel E. Quevedo}
\address{Systems \& Control Engineering, Indian Institute of Technology Bombay, Powai, Mumbai~400076, India}
\thanks{Prabhat K. Mishra and Debasish Chatterjee are with Systems and Control Engineering, Indian Institute of Technology Bombay, Mumbai, India.\email{prabhat@sc.iitb.ac.in, dchatter@iitb.ac.in, dquevedo@ieee.org}}
\thanks{Daniel E. Quevedo is with the Department of Electrical Engineering (EIM-E), Paderborn University,   Germany.}
\thanks{There are minor mistakes in the version \cite{PDQ-LCSS}.}
\keywords{Stochastic predictive control, Kalman filter, constrained control.}
\thanks{}

\begin{document}

\pagenumbering{gobble}

\maketitle

%%%%%%%%%%%%%%%%%%%%%%%%%%%%%%%%%%%%%%%%%%%%%%%%%%%%%%%%%%%%%%%%%%%%%%%%%%%%%%%%
\begin{abstract}

We present a stochastic predictive controller for discrete time linear time invariant systems under incomplete state information. Our approach is based on a suitable choice of control policies, stability constraints, and employment of a Kalman filter to estimate the states of the system from incomplete and corrupt observations. We demonstrate that this approach yields a computationally tractable problem that should be solved online periodically, and that the resulting closed loop system is mean-square bounded for any positive bound on the control actions. Our results allow one to tackle the largest class of linear time invariant systems known to be amenable to stochastic stabilization under bounded control actions via output feedback stochastic predictive control.
\end{abstract}
%\begin{IEEEkeywords}
%Stochastic predictive control, Kalman filter, constrained control.
%\end{IEEEkeywords}

%%%%%%%%%%%%%%%%%%%%%%%%%%%%%%%%%%%%%%%%%%%%%%%%%%%%%%%%%%%%%%%%%%%%%%%%%%%%%%%%
\section{Introduction}
Optimization based control techniques have received tremendous attention because of their wide applicability, tractability, and capability to handle a variety of constraints at the synthesis stage while minimizing some performance objective. Stochastic predictive control (SPC) is an optimization based control technique where the actions are obtained by solving a finite horizon (of length say $N$,) optimal control problem at each sampling instant involving an expected cost given the current state of the plant, where the system dynamics is affected by stochastic uncertainties. The underlying optimization problem yields a control policy \cite[\S 3.4]{ref:May-14},\cite{ mesbah_16_survey}. Receding horizon implementation of SPC then consists of solving the optimization problem every recalculation interval of length $N_r$,  ($N_r \leq N$); the first $N_r$ controls are applied to the plant, the rest are discarded, and the procedure repeated. 
\par This article addresses SPC under output feedback and hard constraints on the control inputs. While there is a growing body of work on SPC with hard control constraints under perfect state information, the output feedback (partial measurements) case is more difficult, but perhaps more important. In most practical applications the entire set of system states cannot be measured, and several predictive control schemes for systems with incomplete state information have been proposed in \cite{mayne2006robust,cannon2012estimation,farina2015,yan2005estimation,ref:Hokayem-12}. In \cite{mayne2006robust, cannon2012estimation}, all the uncertainties are assumed to be bounded, and robust control techniques were adopted. In \cite{farina2015,yan2005estimation}, the process noise is considered to be unbounded but stability and recursive feasibility have not been addressed under hard bounds on the control inputs.  
Since optimization over feedback policies gives, in general, better control performance than that over open loop input sequences \cite{kumar1986stochastic}, performing SPC with policies is recommended. Adopting this approach, when perfect state information is available, saturated disturbance feedback policies were utilized in \cite{ref:ChaHokLyg-11} for SPC. A full solution to the unconstrained SPC problem under output feedback was first provided in \cite{bosgra2003} by proposing innovation feedback, but this work did not involve hard control bounds.\footnote{Recall that the innovation sequence is a quantity in the measurement update equation of Kalman filter, found by the difference between the measured output and the estimated output obtained from the estimated states \cite[page no. 130]{simon2006optimal}.}    
By generalizing the saturated disturbance feedback approach of \cite{ref:ChaHokLyg-11}, a Kalman filter was utilized in \cite{ref:Hokayem-12} under output feedback and bounded controls. Stability of linear systems under Gaussian noise is impossible to ensure with bounded controls, if the spectral radius of the system matrix is greater than unity. The bordering case of the system with a Lyapunov stable (but not asymptotically stable) system matrix is, in general, difficult to analyse. In \cite{ref:ChaHokLyg-11}, the property of  stability of the closed-loop system was proved in the presence of large enough control authority and a Lyapunov stable system matrix. In other words, the optimal control problem in \cite{ref:Hokayem-12} turns out to be feasible only when the hard bound on the control is larger than a threshold that is a function of various statistical objects and the system data. This requirement severely restricts practical  applicability of the controller in \cite{ref:Hokayem-12}, and to overcome this restriction,  here we formulate an alternative controller. 
\par Since stochastic optimal control problems are generally not tractable, we follow the affine saturated innovation feedback policy approach of \cite{ref:Hokayem-12} to get a tractable deterministic surrogate of the underlying stochastic optimal control problem, and employ globally feasible drift conditions to ensure stochastic stability for any positive bound on control.\footnote{Recall that drift conditions \cite{chatterjee-15} relate the values of Lyapunov like functions with their conditional expectations, given the current state, at the next stage of the underlying process.} 
Beyond enlarging the applicability of the main result of \cite{ref:Hokayem-12}, in the current work we present a recursively feasible convex quadratic program (QP) to be solved periodically online as part of the SPC algorithm. 
From an algorithmic perspective, the second order cone program of \cite{ref:Hokayem-12} is replaced by a QP, which has significant numerical advantages \cite{GoldfarbSOCP}.   
\par The remainder of this article is organized as follows: In \secref{s:setup} we establish the definitions of the plant and its properties. Important ingredients of SPC under output feedback are described in \secref{s:output feedback}. We discuss stability issues in \secref{s:stability}. In \secref{s:main results} we provide our main result 
that is validated via numerical experiments in \secref{s:simulation}. We conclude in \secref{s:conclusion} and present our proofs in the Appendix.
\par Let $\R, \Nz$, and $\N$ denote the set of real numbers, the non-negative integers, and the positive integers, respectively.  \(I_d\) is the \(d\times d\) identity matrix and \( \zeros \) is the matrix of appropriate dimension with 0 entries. 
For given $\zeta$ and $r$, we define the component-wise saturation function  \(\R^{\nu}\ni z\longmapsto \sat_{r, \zeta}^\infty(z)\in\R^{\nu}\) to be 
			\[
				\bigl(\sat_{r, \zeta}^\infty(z)\bigr)^{(i)} = \begin{cases}
					z^{(i)} \zeta/r	& \text{if \(\abs{z^{(i)}} \le r\),}\\
					\zeta	 	& \text{if \(z^{(i)} > r\), and }\\
					-\zeta		& \text{otherwise,}
				\end{cases}
			\]for each \(i = 1, \ldots, \nu\).
For any sequence \((s_n)_{n\in\Nz}\) taking values in some Euclidean space, we denote by \( s_{n:k} \) the vector \(\bmat{s_n\transp & s_{n+1}\transp & \cdots & s_{n+k-1}\transp}\transp\), \(k\in\Nz\). The notation \(\EE_z[\cdot]\) stands for the conditional expectation with given \(z\). For a given vector \(C\), its $i^{th}$ component is denoted by $C^{(i)}$. 
Similarly, $C^{(j,:)}$ denote the $j^{th}$ row of a given matrix $C$. 
$\sigma_1(M)$ denotes the largest singular value of $M$, and $M^+$ is the Moore-Penrose pseudo-inverse of $M$ \cite[\S 6.1]{bernstein2009matrix}. For $\xi \in \R$ we let $\xi_+ \Let \max\{0,\xi\}$, and $\xi_- \Let \max\{0,-\xi\}$. Let $\reachab_{N}(\A, \B)$ denote the matrix $\bmat{\A^{N-1}\B & \A^{N-2}\B & \cdots & \B}$. We let $]a,b[ \; \Let \{z \in \R \mid a < z < b \}$. 

\section{Dynamics and objective function}\label{s:setup}
Consider a discrete time dynamical system 
\begin{subequations}\label{e:system}
\begin{align}
\st_{t+1} &= \A\st_t + \B\control_t + \wnoise_t  \label{e:steq} \\ 
\meas_t & = \C \st_t + \mnoise_t,
\end{align}
\end{subequations}
where $t \in \Nz$, and $\st_t \in \R^d$, $\control_t \in \R^m$, $\meas_t \in \R^q$ are the states, the control inputs, and the outputs, respectively, at time $t$. The process noise $\wnoise_t \in \R^d$ and the measurement noise $\mnoise_t \in \R^q$ are stochastic processes, the system matrices $\A,\B$ and $\C$ are known and are of appropriate dimensions. The control $\control_t$ is constrained as per:
				\begin{equation}\label{e:controlset}
				\control_t \in \controlset\Let\{v\in\R^m\mid \norm{v}_\infty \leq \authority \} \text{ for all } t.
				\end{equation}				
We have the following assumptions:				
\begin{enumerate}[label={\rm (A\arabic*)}, leftmargin=*, widest=3, align=left, start=1, nosep]
\item The pair $(\A,\B)$ is stabilizable and the pair $(\A,\C)$ is observable.  \label{a:stabilizability}
\item The initial condition, the process and the measurement noise vectors are mutually independent and normally distributed with $\st_0 \sim N(0,\Sigma_{\st_0})$, $\wnoise_t \sim N(0,\Sigma_{\wnoise}), \mnoise_t \sim N(0,\Sigma_{\mnoise})$, with $\Sigma_{\st_0} \succeq 0$, $ \Sigma_{\wnoise} \succ 0$ and $\Sigma_{\mnoise} \succ 0$. \label{as:normal distribution}
\item The system matrix $\A$ is Lyapunov stable. \footnote{ Recall that a Lyapunov stable matrix has all its eigenvalues in the closed unit disk, and those on the unit circle have equal algebraic and geometric multiplicities. The Assumption \ref{as:Lyapunov} is not needed for algorithmic tractability of our results, but is crucial for stability. In fact it is known \cite{ref:ChaRamHokLyg-12} that stability of \eqref{e:system} under bounded controls is impossible to ensure, if the spectral radius of $\A$ is larger than unity.} \label{as:Lyapunov}
\item $(\A, \Sigma_{\wnoise}^{1/2})$ is controllable.
\end{enumerate}				
\par For each $t$ let $ \YY_t \Let \{\meas_0, \cdots, \meas_t \} $ denote the set of observations upto time $t$. Let us fix an optimization horizon $N \in \N$ and recalculation interval $N_r \in \N$ such that $N_r \leq N$. We define the cost $V_t$ as
				\begin{equation}\label{e:cost}
				V_t \Let \EE_{\YY_t} \left[ \sum_{k = 0}^{N-1} (\norm{\st_{t+k}}^2_{Q_k} + \norm{\control_{t+k}}^2_{R_k}  ) + \norm{\st_{t+N}}^2_{Q_N}\right],
				\end{equation}
where, $Q_k, R_k, Q_N$ are given symmetric positive semi-definite matrices of appropriate dimensions, for $k = 0,\cdots,N-1$.	
The evolution of the system \eqref{e:system} over one optimization horizon can be described below in a compact form as:
\begin{subequations}\label{e:stacked dynamics}
\begin{align}
\st_{t:N+1} &= \calA\st_t + \calB\control_{t:N} + \calD \wnoise_{t:N}, \label{e:stacked state}\\
 \meas_{t:N+1} &= \calC \st_{t:N+1} + \mnoise_{t:N+1},
\end{align}
\end{subequations}
where $\calA$, $\calB$, $\calC$ and $\calD$ are standard matrices of appropriate dimensions. 
Similarly, the cost function \eqref{e:cost} can also be written in a compact form as:
		\begin{equation}\label{e:cost compact}
		V_t = \EE_{\YY_t} \left[ \norm{\st_{t:N+1}}^2_{\calQ} + \norm{\control_{t:N}}^2_{\calR} \right],
		\end{equation}	
where $\calQ$ and $\calR$ are standard block diagonal matrices.	
The following optimal control problem constitutes the backbone of SPC under output feedback:
\begin{equation}
		\label{e:opt control problem}
		\begin{aligned}
			& \minimize_{\control_{t:N}}	&&  \text{ objective function }\eqref{e:cost}\\
			& \sbjto	 & &  \begin{cases}				
				\text{system dynamics }\eqref{e:stacked state},\\
				\text{hard constraint on control } \eqref{e:controlset}, \\
				\text{control policy class},\\				
				\text{drift conditions},
			\end{cases}
		\end{aligned}
		\end{equation}
where we shall impose a specific control policy class and drift conditions in \secref{s:output feedback} and \secref{s:stability}, respectively.
Our SPC consists of solving \eqref{e:opt control problem} every $N_r$ time steps; see Fig.\ \ref{Fig:recalculation}. 		
Our choice of control policies will ensure algorithmic tractability and drift conditions will ensure stability of closed-loop system.
\begin{figure}[t]
\centering
\begin{adjustbox}{width = \columnwidth}
\includegraphics{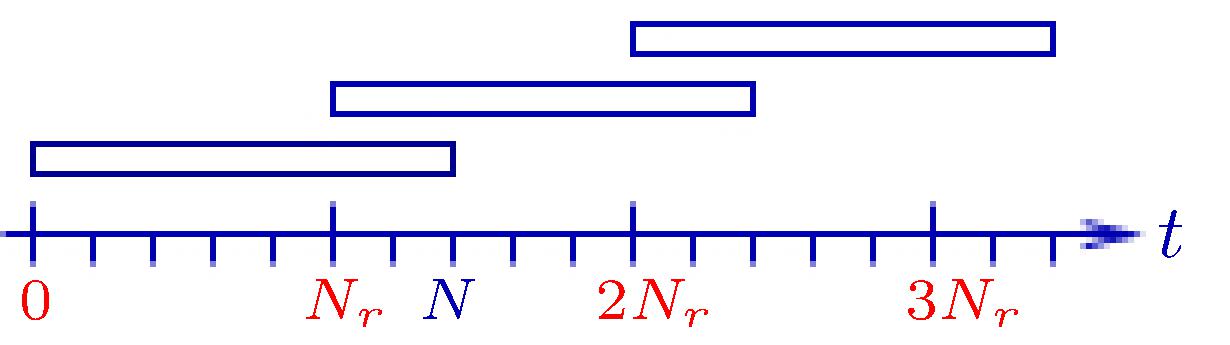}
\end{adjustbox}
\caption{Receding horizon control strategy: At $t=0$, $N$ future control commands are computed but only first $N_r$ of them are applied; the process repeats after every $N_r$ time steps.}
\label{Fig:recalculation}
\end{figure}
\section{Output feedback policies}\label{s:output feedback}
In this section we describe two important ingredients of SPC in the presence of incomplete state information. First of all, we need to estimate the states of the system given its output, and secondly, choose a suitable policy in terms of saturated innovations \cite{bosgra2003}. 
\subsection{Estimator}
For $t,s \in \Nz$ $t \geq s$ let us define $\hat{\st}_{t \mid s} \Let \EE_{\YY_s}\left[ \st_t \right]$ and \[ P_{t \mid s} \Let \EE_{\YY_s} \left[ (\st_t - \hat{\st}_{t \mid s}) (\st_t - \hat{\st}_{t \mid s}) \transp \right]. \] For simplicity of notation we denote $\hat{\st}_{t \mid t}$ by $\hat{\st}_t$ and $P_{t \mid t}$ by $P_t$. We need a fundamental result related to Kalman filtering, which is stated in \cite[p.102]{kumar1986stochastic}.  
 Let us define the estimation error vector $e_t \Let \st_t - \hat{\st}_t$. 
 Let \[ K_t \Let (\A P_t \A \transp + \Sigma_{\wnoise})\C \transp \Bigl(\C(\A P_t \A \transp + \Sigma_{\wnoise})\C\transp + \Sigma_{\mnoise} \Bigr)\inverse,\] \[
  \Gamma_t \Let I_d - K_t \C, \text{ and } \phi_t \Let \Gamma_t \A. \] 
  The filter equations in \cite[p.102]{kumar1986stochastic} are used to find the evolution of the estimator error over one optimization horizon as follows: 
\begin{equation} \label{e:est error}
e_{t:N+1} = \calF_t e_t + \calG_t \wnoise_{t:N} - \calH_t \mnoise_{t:N+1},
\end{equation}
where $\calF_t \Let \bmat{I_d \\ \phi_t \\ \phi_{t+1}\phi_t  \\ \vdots \\ \phi_{t+N-1}\cdots \phi_t }$, 
$\calG_t \Let \bmat{0 & \cdots & 0 & 0 \\ \Gamma_t & \cdots & 0 & 0 \\ \phi_{t+1}\Gamma_t & \cdots & 0 & 0 \\ \vdots & \cdots & \vdots & \vdots \\ \phi_{t+N-2} \cdots \phi_{t+1}\Gamma_t & \cdots & \Gamma_{t+N-1} & 0 \\ \phi_{t+N-1}\cdots \phi_{t+1}\Gamma_t & \cdots & \phi_{t+N-1}\Gamma_{t+N-2} & \Gamma_{t+N-1} }$, 
$\calH_t \Let \bmat{0 & 0 & \cdots & 0 & 0 \\0 & K_t & \cdots & 0 & 0 \\ 0 & \phi_{t+1} K_t & \cdots & 0 & 0 \\ 0 & \vdots & \cdots & \vdots & \vdots \\ 0 & \phi_{t+N-2} \cdots \phi_{t+1} K_t & \cdots & K_{t+N-1} & 0  \\ 0 & \phi_{t+N-1}\cdots \phi_{t+1} K_t & \cdots & \phi_{t+N-1} K_{t+N-2} & K_{t+N-1} }$. 
Since the optimization is done every $N_r$ time steps, we find $\hat{\st}_{t+N_r}$ by using \eqref{e:system} and the filter equations \cite[p.102]{kumar1986stochastic} as:
\begin{equation}\label{e:estimator st}
\hat{\st}_{t+N_r} = \A^{N_r}\hat{\st}_t + \reachab_{N_r}(\A,\B)\control_{t:N_r} + \Xi_{t+N_r-1},
\end{equation}
where
		\begin{equation*}
\begin{aligned} 
&\Xi_{t+ N_r -1} = \bmat{\A^{N_r-1}K_t\C\A & \cdots & K_{t+N_{r} -1}\C\A }e_{t:N_r} \\
& + \bmat{\A^{N_r-1}K_t\C & \A^{N_r-2}K_{t+1}\C & \cdots & K_{t+N_{r} -1}\C } \wnoise_{t:N_r} \\
& + \bmat{\A^{N_r-1}K_t & \A^{N_r-2}K_{t+1} & \cdots & K_{t+N_{r} -1} } \mnoise_{t+1:N_r} .
\end{aligned}
\end{equation*}
The quantity $\Xi_{t+N_r-1}$ in \eqref{e:estimator st} is such that $\EE_{\YY_t} \left[ \Xi_{t+N_r-1} \right] = \zeros $ and it admits a bound:
\begin{proposition}(\cite[Proposition 3]{ref:Hokayem-12})
\label{prop:bound on Xi}
There exists some $T^{\prime}, \beta > 0$ such that  
 $ \EE_{\YY_t}\left[ \norm{\Xi_{t+N_r-1}} \right] \leq \beta $ for all $t \geq T^{\prime}$.
\end{proposition}
\subsection{Control policy class}
\par We select the affine parametrization of control policies \cite{skaf2010,bertsimas2010optimality, bertsimas2007constrained} under output feedback. It is demonstrated in \cite{bosgra2003} that, in the absence of control bounds, optimization over innovation feedback $(\meas_t-\hat{\meas}_t)$ leads to convex problems. The innovation sequence for one optimization horizon is then given by:
\begin{equation}
\label{e:innovation}
\meas_{t:N+1} - \hat{\meas}_{t:N+1} = \calC \calF_t e_t + \calC \calG_t \wnoise_{t:N} +(I - \calC \calH_t)\mnoise_{t:N+1},
\end{equation} 
where $\hat{\meas}_{t:N+1} = \calC \hat{\st}_{t:N+1}$, $\calF_t, \calG_t$ and $\calH_t$ are as in \eqref{e:est error}, and $\calC$ is defined in \eqref{e:stacked dynamics}. However, this policy is inadmissible because controls are bounded. To satisfy hard bounds on the controls while retaining computational tractability, we consider affine parametrization in terms of the saturated values of innovation feedback: We periodically minimize the cost \eqref{e:cost} over the following causal feedback policy for $\ell = 0,\cdots , N-1$, 
\begin{equation}\label{e:policyclass}
\control_{t+\ell} = \offset_{t+\ell} + \sum_{i=0}^{\ell} \theta_{\ell,t+i}\ee_i (\meas_{t+i} - \hat{\meas}_{t+i}),
\end{equation}
where \(\ee_i:\R^q\lra\ \R^q\) is a measurable map for each \(i\) such that  $\norm{\ee_i(\meas_{t+i}-\hat{\meas}_{t+i})}_{\infty} \leq \ee_{\max}$.  The above control policy class \eqref{e:policyclass} can be represented in a compact form as follows:
\begin{equation} \label{e:stacked control}
\control_{t:N} = \offset_t + \gain_t \ee(\meas_{t:N+1}-\hat{\meas}_{t:N+1})
\end{equation} 
where \(\offset_t\in\R^{m N}\) $\ee \Let \bmat{\ee_0\transp & \cdots & \ee_{N}\transp}$ and \(\gain_t\) is the following block triangular matrix
		\begin{equation*} \label{e:gain}
			\gain_t = \bmat{ \theta_{0, t} & \zeros & \cdots & \zeros & \zeros & \zeros\\ \theta_{1, t} & \theta_{1, t+1}  & \cdots & \zeros & \zeros & \zeros\\ \vdots & \vdots & \vdots & \vdots & \vdots & \vdots\\ \theta_{N-1, t} & \theta_{N-1, t+1} & \cdots & \theta_{N-1, t+N-2} & \theta_{N-1, t+N-1} & \zeros},
		\end{equation*}
		with each \(\theta_{k, \ell} \in \R^{m\times q}\) and  \[ \norm{\ee(\meas_{t:N+1}-\hat{\meas}_{t:N+1})}_{\infty} \leq \ee_{\max}. \] 
We choose $\ee_i$'s to be component-wise odd functions, e.g., standard saturation function, sigmoidal function, etc. 
%Since the innovation sequence is mean zero Gaussian, it holds that $\EE_{\YY_t}\left[\ee_i(\meas_{t+i} - \hat{\meas}_{t+i}) \right] = \zeros$ for each $i = 1, \ldots N$.
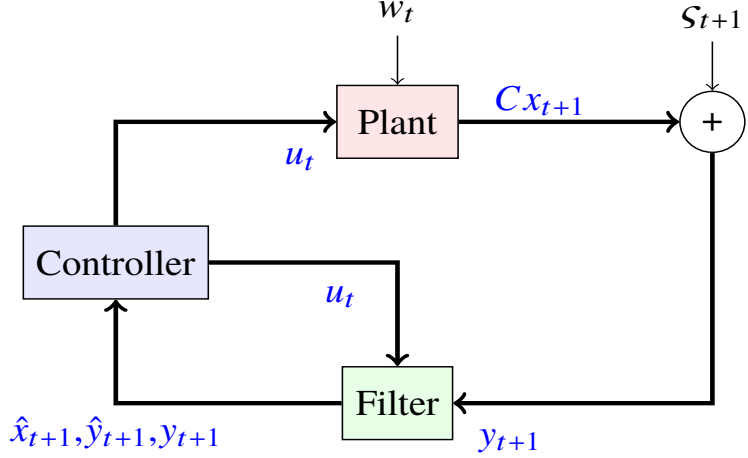
\begin{figure}
\centering
\begin{adjustbox}{width = 0.8\columnwidth, height = 0.5\columnwidth}
\begin{tikzpicture}
\tikzstyle{pinstyle} = [pin edge={to-,thin,black}]	
\tikzstyle{block} = [draw, fill=blue!10, rectangle, 
					    minimum height=2em, minimum width=1]
\tikzstyle{blockgreenl} = [draw, fill=green!10, rectangle, 
					    minimum height=2em, minimum width=1cm]
\tikzstyle{blockgreen} = [draw, fill=green!10, rectangle, 
					    minimum height=2em, minimum width=1cm] 
\tikzstyle{blockred} = [draw, fill=red!10, rectangle, 
					    minimum height=2em, minimum width=1.1cm] 
\tikzstyle{sum} = [draw, circle]					     
%===================================================
\node[coordinate] (0) at (0,0) {};
\node [blockred, right= 2cm of 0, pin={[pinstyle]above:$\wnoise_t$}] (plant) {Plant};			%\draw[help lines] (0,0) grid (7,3);
\node[coordinate, right= 2cm of plant] (1) {1};
\node [blockgreen, below= 2cm of plant] (filter) {Filter};
\node [block, below= 1cm of 0] (controller) {Controller};
\node [sum, right= 2cm of plant, pin={[pinstyle]above:$\mnoise_{t+1}$}] (sum) {+};
\node[anchor=west, below= 1cm of controller.south, blue] {$\hat{\st}_{t+1}$,$\hat{\meas}_{t+1}$,$\meas_{t+1}$};
\node[anchor=west, left= 0.6cm of plant.south, blue] {$\control_t$};
\node[anchor=west, right= 0.6cm of filter.south, blue] {$\meas_{t+1}$};
\node[anchor=west, above= 0.8cm of filter.west, blue] {$\control_t$};
\node[anchor=north, right= 0.2cm of plant.20, blue] {$C \st_{t+1}$};
%==========================================
\draw[very thick ,->] (plant.east) --  (sum.west);
\draw[very thick ,->] (controller.north) |- (plant.west);
\draw[very thick ,->] (sum.south) |- (filter.east);
\draw[very thick ,->] (filter.west) -| (controller.south);
\draw[very thick ,->] (controller.east) -| (filter.north);
\end{tikzpicture}
\end{adjustbox}	
\caption{The Kalman filter employs the measurement $\meas_t$ to estimate the state $\hat{\st}_t$ and is aware of the control $\control_t$. }				
\end{figure}
\section{Stability}\label{s:stability}
It is well known that the construction of a robust positively invariant set is not possible in the presence of Gaussian noise \cite[\S 3.4]{ref:May-14}. Therefore, standard deterministic Lyapunov based arguments for proving stability are not applicable. Moreover, a linear stochastic system cannot be globally stabilized with the help of bounded controls {\cite[Theorem 1.7, Open problem 1.3]{ref:ChaRamHokLyg-12}} if the nominal plant has spectral radius larger than unity. Under this fundamental limitation, we restrict our attention to Lyapunov stable plants for ensuring stability. 
We establish the property of mean-square boundedness of the closed-loop plant under \ref{as:Lyapunov}. This class of systems is indeed, till date, the largest class of linear time invariant systems that can be globally stabilized with bounded control actions. We recall the following definition:
\begin{definition}\cite[\S III.A]{chatterjee-15}
\rm{
An $\R^d$-valued random process $(\st_t)_{t \in \Nz_0}$ with given initial measurement of output $\YY_0$ is said to be \emph{mean-square bounded} if \[ \sup_{t \in \Nz}\EE_{\YY_0}[\norm{\st_t}^2] < +\infty . \]  
}
\end{definition}   
\par Note that a Lyapunov stable system matrix $\A$ can be decomposed \cite{ref:HokChaRamChaLyg-10} into a Schur stable part $\A_s$ and an orthogonal part $\A_o$, resulting in recursion \eqref{e:estimator st} as follows: 
\begin{equation}\label{e:decomposed st}
\bmat{\hat{\st}_{t+1}^o \\ \hat{\st}_{t+1}^s} = \bmat{\Aortho \hat{\st}_{t}^o \\ \Aschur \hat{\st}_{t}^s} + \bmat{\Bortho  \\ \Bschur} \control_t + \bmat{\Xi_t^o \\ \Xi_t^s}  ,
\end{equation}
where ${\hat{\st}}_t^s \in \R^{\dschur}$ , ${\hat{\st}}_t^o \in \R^{\dortho}$, $d = \dortho + \dschur $, and $\Xi_t$ is defined in \eqref{e:estimator st}.
Linear systems with Schur stable matrices are automatically stable (in suitable sense) under bounded controls, but stability of orthogonal systems in presence of stochastic noise is \emph{not obvious}. 
The following stability constraint for orthogonal subsystem was presented in \cite{ref:Hokayem-12}: 
\begin{align}\label{e:drift enough control}
\begin{aligned}
\norm{\Aortho^{\reachindex}\hat{\st}^o_t + \reachab_{\reachindex}(\Aortho,\Bortho)\control_{t:\reachindex}} & \leq \norm{\hat{\st}^o_t} - (\beta + \dfrac{\varepsilon^\prime}{2}  )  \\ 
\text{ whenever } \norm{\hat{\st}^o_t} &  \geq \beta + \varepsilon^{\prime}, 
\end{aligned}
\end{align}
where $\varepsilon^{\prime} > 0$, $\beta$ is as defined in Proposition \ref{prop:bound on Xi}, and $\reachindex$ is the \emph{reachability index} of the matrix pair $(\Aortho, \Bortho)$.
Mean square boundedness of the closed-loop system with the above stability constraint was proved in \cite{ref:Hokayem-12} under large enough control authority. In particular, it was assumed in \cite{ref:Hokayem-12} that  \[ \authority \geq  \sigma_1 \left( \reachab_{N_r}(\Aortho,\Bortho) ^+\right) (\beta + \dfrac{\varepsilon^{\prime}}{2}). \] Let us consider the following example. 
\begin{example}
\rm{
Let us consider the dynamics \eqref{e:decomposed st} when $\Aortho=1, \dschur = 0$:
\begin{align} \label{e:scalar example}
\begin{aligned}
& \hat{\st}_{t+1}^o = \hat{\st}_t^o + \control_t + \Xi_t^o, \\
& \hat{\meas}^o_t = \hat{\st}_t^o, \quad \meas_t^o = \st_t^o + \mnoise_t^o .   
\end{aligned}
\end{align}
Then the drift condition \eqref{e:drift enough control} given in \cite{ref:Hokayem-12} requires a scalar control such that $\abs{\hat{\st}_t^o + \control_t} - \abs{\hat{\st}_t^o} \leq  - (\beta + \dfrac{\varepsilon^{\prime}}{2})$ for all $\hat{\st}_t^o$. Here equality holds when $\control_t$ has magnitude $\beta + \dfrac{\varepsilon^{\prime}}{2}$ and sign opposite to that of $\hat{\st}_t^o$. Since $\beta$ depends on the bound on the uncertainty (see Proposition \ref{prop:bound on Xi}), stability does not follow from \cite[Theorem 1]{ref:Hokayem-12} when $\authority < \beta + \dfrac{\varepsilon^{\prime}}{2}$.
$\hfill\square$
}
\end{example}
However, this lower bound on the control authority is artificial and we show that it can be removed if different drift conditions are employed, resulting in different stability constraints. 
%\par In this article we remove this lower bound by using component-wise drift conditions. 
We have the following Lemma.
\begin{lemma}
\label{lem:ortho stable}
Consider the orthogonal part of the system \eqref{e:decomposed st}. 
If there exists a \(\reachindex\)-history dependent policy such that for any given $r > 0, \text{ and } 0 < \zeta < \frac{\authority}{\sqrt{d_o}\sigma_{1}\left(\reachab_{\reachindex}(\Aortho,\Bortho)^{+}\right)}, $ and for any $ t=0, \reachindex, 2 \reachindex, ...$, the control $\control_{t:\reachindex} \in \controlset^{\reachindex}$ is chosen such that for $ j=1,2,\cdots ,d_o$ the following conditions hold:
\begin{subequations}\label{e:drift}
\begin{align}
\EE_{\YY_t} \left[ \left( (\Aortho^{\reachindex})\transp \reachab_{\reachindex}(\Aortho, \Bortho)\control_{t:\reachindex} \right)^{(j)} \right] &\leq -\zeta 
\text{ if }  \left( \hat{\st}^o_{ t} \right)^{(j)}  > r, \text{ and} \label{e:drift1} \\
\EE_{\YY_t} \left[ \left( (\Aortho^{\reachindex})\transp \reachab_{\reachindex}(\Aortho, \Bortho)\control_{t:\reachindex}\right)^{(j)} \right] & \geq \zeta 
\text{ if } \left( ( \hat{\st}^o_{ t}  \right)^{(j)}  < -r . \label{e:drift2}
\end{align}
\end{subequations}
Successive application of this policy renders the orthogonal part of the closed-loop system \eqref{e:decomposed st} mean-square bounded; i.e. there exist $T, \gamma_o > 0$ such that $\EE_{\YY_T} \left[ \norm{\hat{\st}_t^o}^2 \right] \leq \gamma_o$ for all $t \geq T$.
\end{lemma} 
A proof of Lemma \ref{lem:ortho stable} is given in the appendix. \\
In order to implement the drift conditions \eqref{e:drift} with the help of a tractable optimization program, we consider the first $\reachindex$ blocks of \eqref{e:stacked control} 
\begin{equation} \label{e:kappa blocks control}
\control_{t:\reachindex} = (\offset_t)_{1:\reachindex m} + (\gain_t)_{1:\reachindex m } \ee(\meas_{t:N+1}-\hat{\meas}_{t:N+1})
\end{equation} 
and substitute them in \eqref{e:drift}. Let us separate the first $q$ columns of the gain matrix $\gain_t$ and the first $q$ rows of the saturated innovation sequence such that $\gain_t = \bmat{\gain_t^{(:,t)} & \gain_t^{\prime}} $ and $\ee(\meas_{t:N+1}-\hat{\meas}_{t:N+1}) = \bmat{\ee_0(\meas_t - \hat{\meas}_t) \\ \ee^{\prime}(\meas_{t+1:N} - \hat{\meas}_{t+1:N})}$, where $\gain_t^{(:,t)}$ represents the first $q$ columns of $\gain_t$, other notations are clear from the context. We get the following stability constraints for $ \varepsilon > 0$:
%\begin{align}
%& \quad \Biggl( (\Aortho^{ \reachindex})\transp\reachab_{\reachindex}(\Aortho, \Bortho) \left( (\offset_t)_{1: \reachindex m} + (\gain_t^{(:,t)})_{1:\reachindex m}\ee_0(\noisyInnovation_t) \right) \Biggr)^{(j)} \leq -\zeta \nonumber  \\
%& \quad \quad \text{ whenever } \left( \stest^o_{t}  \right)^{(j)} \geq r + \varepsilon, \label{e:stability constraint 1}\\
%& \quad \Biggl( (\Aortho^{\reachindex})\transp\reachab_{\reachindex}(\Aortho, \Bortho)\left( (\offset_t)_{1: \reachindex m} + (\gain_t^{(:,t)})_{1:\reachindex m}\ee_0(\noisyInnovation_t) \right) \Biggr)^{(j)} \geq \zeta \nonumber \\ 
%& \quad \quad \text{ whenever } \left( \stest^o_{t} \right)^{(j)} \leq -r - \varepsilon, \label{e:stability constraint 2} 
%\end{align}
\begin{subequations}\label{e:stability constraint}
\begin{align}
& \quad \Biggl( (\Aortho^{ \reachindex})\transp\reachab_{\reachindex}(\Aortho, \Bortho) \left( (\offset_t)_{1: \reachindex m} + (\gain_t^{(:,t)})_{1:\reachindex m}\ee_0(\meas_t - \hat{\meas}_t) \right) \Biggr)^{(j)} \leq -\zeta \nonumber \\
 & \quad \quad \text{if }  \left( \hat{\st}^o_{t} \right)^{(j)} \geq r + \varepsilon, \text{ and } \label{e:stability constraint 1}\\
& \quad \Biggl( (\Aortho^{\reachindex})\transp\reachab_{\reachindex}(\Aortho, \Bortho)\left( (\offset_t)_{1: \reachindex m} + (\gain_t^{(:,t)})_{1:\reachindex m}\ee_0(\meas_t - \hat{\meas}_t) \right) \Biggr)^{(j)} \geq \zeta \nonumber \\  
& \quad \quad \text{if }  \left( \hat{\st}^o_{t} \right)^{(j)} \leq -r - \varepsilon, \label{e:stability constraint 2} 
\end{align}
\end{subequations}
where $r, \zeta$ and $j$ are as in \eqref{e:drift}.
\begin{example}
\rm{
Let us consider the dynamics \eqref{e:scalar example}. The conditions \eqref{e:stability constraint} provide expected drift towards the interval $]-r-\varepsilon, r+\varepsilon[$ in one step whenever $\hat{\st}_t^o$ is outside that interval. The amount of constant negative drift $\zeta$ can be selected from the interval $]0, \authority [$ to respect the hard bound on control.
}
$\hfill\square$
\end{example}   
\section{Main results}\label{s:main results}
In this section we recast the constrained optimal control problem \eqref{e:opt control problem} as a convex quadratic optimization program that is recursively feasible, and its receding horizon implementation ensures mean-square boundedness of the system states. We can select the recalculation interval $N_r$ such that $\reachindex \leq N_r \leq N$. For simplicity, we choose $N_r = \reachindex \leq N$. Let us define $\Pi_{\meas_t} = \ee_0(\meas_t - \hat{\meas}_t) \ee_0(\meas_t - \hat{\meas}_t) \transp $, $\Sigma_{\ee} = \EE\Bigl[\ee^{\prime}(\meas_{t+1:N} - \hat{\meas}_{t+1:N})\ee^{\prime}(\meas_{t+1:N} - \hat{\meas}_{t+1:N})\transp \Bigr]$, $\Sigma_{\ee^{\prime}\wnoise} = \EE \Biggl[\ee^{\prime}(\meas_{t+1:N} - \hat{\meas}_{t+1:N})\wnoise_{t:N} \transp \Biggr]$, $\Sigma_{e\ee^{\prime}} = \EE \Bigl[ \ee^{\prime} (\meas_{t+1:N} - \hat{\meas}_{t+1:N})e_t \transp \Bigr]$
%$\Sigma_{\ee\wnoise} \Let  \EE_{\YY_t} \left[ \wnoise_{t:N} \ee(\meas_{t:N+1} - \hat{\meas}_{t:N+1})\transp \right]$, 
%\Sigma_{\ee} \Let \EE_{\YY_t} \left[ \ee(\meas_{t:N+1} - \hat{\meas}_{t:N+1})\ee(\meas_{t:N+1} - \hat{\meas}_{t:N+1})\transp \right],
%\Sigma_{e\ee} \Let \EE_{\YY_t} \left[ e_t \ee(\meas_{t:N+1} - \hat{\meas}_{t:N+1})\transp \right],$
and $\alpha \Let \calB\transp\calQ\calB + \calR$. We have the following theorems:
\begin{theorem}\label{th:main}
For every time $t= 0, N_r, 2N_r, \cdots$, the optimal control problem \eqref{e:opt control problem} can be written as the following convex quadratic, (globally) feasible program:
\begin{align}
\minimize_{\offset_t,\gain_t} & \quad \offset_t\transp \alpha \offset_t + \trace(\alpha \gain_t^{(:,t)} \Pi_{\meas_t} (\gain_t^{(:,t)})\transp) + \trace (\alpha \gain_t^{\prime} \Sigma_{\ee} (\gain_t^{\prime})\transp) +2\hat{\st}_t \transp \calA \transp \calQ \calB\offset_t \notag \\
& +  2(\offset_t \transp \alpha + \hat{\st}_t \transp \calA \transp \calQ \calB ) \gain_t^{(:,t)}\ee_0(\meas_t - \hat{\meas}_t) + 2\trace \Bigl( \calD \transp \calQ \calB\gain_t^{\prime} \Sigma_{\ee^{\prime}\wnoise} +  \calA\transp \calQ \calB \gain_t^{\prime} \Sigma_{e \ee^{\prime}} \Bigr)  .   \label{e:obj} \\ %
\sbjto & \quad \abs{\offset_t^{(i)}} + \norm{\gain_t^{(i,:)}}_1\ee_{\max} \leq \authority \; \forall i = 1,\cdots,Nm,  \label{e:hard bound on control} \\
& \quad \text{stability constraints } \eqref{e:stability constraint}.
\end{align}
%\begin{align}
%\minimize_{\offset_t,\gain_t} & \quad \offset_t \transp \alpha \offset_t + \trace(\gain_t \transp \alpha \gain_t \Sigma_{\ee}) + 2\offset_t \transp \calB\transp \calQ \calA \hat{\st}_t \notag \\ 
%& \quad + 2 \trace\Bigl(\gain_t \transp \calB \transp \calQ \left( \calD \Sigma_{\wnoise\ee} + \calA\Sigma_{e\ee} \right) \Bigr)   \label{e:obj} \\ %
%\sbjto & \quad \abs{\offset_t^{(i)}} + \norm{\gain_t^{(i,:)}}_1\ee_{\max} \leq \authority \; \forall i = 1,\cdots,Nm,  \label{e:hard bound on control} \\
%& \quad \text{stability constraints } \eqref{e:stability constraint}.
%\end{align}
\end{theorem}
\par The matrices $\Sigma_{\ee\wnoise}$, $\Sigma_{\ee}$ and $\Sigma_{e\ee}$ required in the above optimization program depend on $\meas_{t:N+1}-\hat{\meas}_{t:N+1}$, which depend on the time variant quantity $P_t$. Since $P_t$ converges asymptotically to a stationary value \cite[\S 4.4]{ref:Hokayem-12}, $\Sigma_{\ee\wnoise}$, $\Sigma_{\ee}$ and $\Sigma_{e\ee}$ can be easily computed offline empirically via classical Monte Carlo methods \cite{ref:robert-13}. Computations for determining our policy were carried out in the MATLAB-based software package YALMIP \cite{ref:lofberg-04}, and were solved using SDPT3-4.0 \cite{ref:toh-06}. 

\begin{theorem}\label{claim:stbl}
The successive applications of the controls given by the optimization program in Theorem \ref{th:main} above renders the closed-loop system mean-square bounded for any positive bound on control. 
\end{theorem}
\par Proofs of Theorem \ref{th:main} and Theorem \ref{claim:stbl} are provided in the appendix.

\section{Numerical Experiment}\label{s:simulation}
We present numerical experiments to illustrate our results and compare our approach against \cite{ref:Hokayem-12} with same objectiive function but different stability constraints. Let us consider the four dimensional linear stochastic system \eqref{e:system} with matrices lifted from \cite{ref:Hokayem-12}:
\[ \A = \bmat{0.9 & 0 & 0 & 0\\ 0 & 1 & 0 & 0\\ 0 & 0 & 0 & -1\\ 0 & 0 & 1 & 0}, \B = \bmat{0 \\ 1\\ 0\\ 1}, C=I_4. \] The simulation data was chosen to be $\st_0 \sim N(0,I_4)$, $\wnoise_t \sim N(0,I_4)$, $\mnoise_t \sim N(0,I_4)$. 

We solved a constrained finite-horizon optimal control problem corresponding to states and control weights \( Q = I_4 , Q_f = I_4, R =1 \). We selected an optimization horizon, \(N=5\), recalculation interval \(N_r = \reachindex = 3 \) and simulated the system responses. We selected the nonlinear bounded term $\ee(\cdot)$ in our policy to be a vector of scalar sigmoidal functions \[ \varphi(\xi)=\frac{1- e^{-\xi}}{1+ e^{-\xi}} \] applied to each coordinate of innovation sequence. We plot empirical mean square bound with respect to  $\authority$ picked from the set $\{0.1, 0.5, 1,2,3,4,5,10,20\}$. All the averages are taken over 100 sample paths for 90 time steps. 
The optimization program of \cite{ref:Hokayem-12} becomes infeasible when $\authority \leq 3$, because the stability constraints used in \cite{ref:Hokayem-12} require a larger bound on the controls. Therefore, we modified the optimization algorithm of \cite{ref:Hokayem-12}. For our purposes we have forced the control values to $\zeros$ whenever the termination code of the solver is not equal to $0$. Our proposed controller performs better than this modified controller for $\authority = 0.1,0.5,1,2$, and yields similar performance for $\authority = 3$. For $\authority \geq 4$, all three controllers perform similarly in terms of the mean square bound of the closed-loop states.   

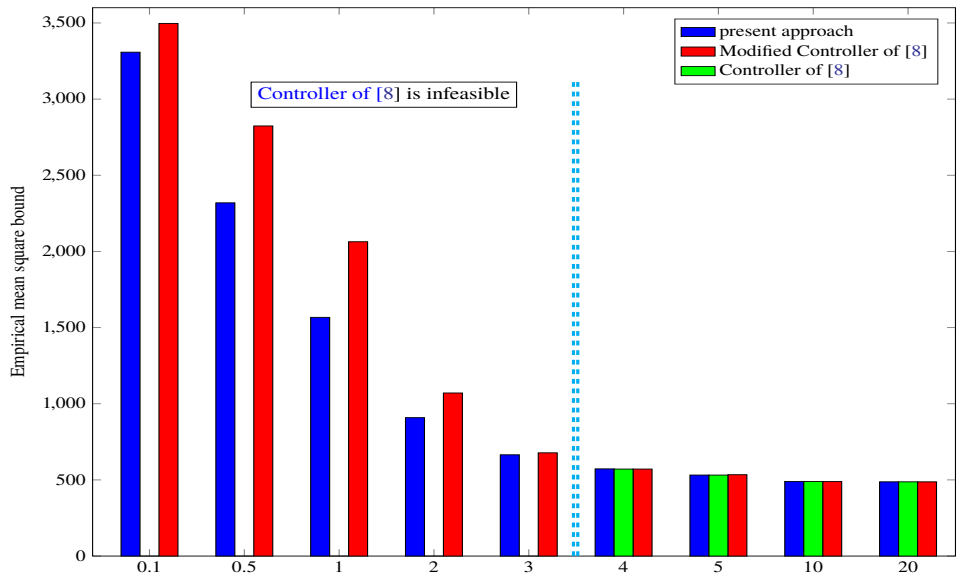
\begin{figure}[b]
	\begin{adjustbox}{width=\columnwidth, height = 0.6\columnwidth}
		\begin{tikzpicture}		
		\begin{axis}[%
		width=6.028in,
		height=4.754in,
		at={(1.011in,0.642in)},
		scale only axis,
		xmin= 0.4,
		xmax=9.5,
		xtick={1,2,3,4,5,6,7,8,9},
		xticklabels={{0.1},{0.5},{1},{2},{3},{4},{5},{10},{20}},
	%	xmajorgrids,
		ymin=0,
		ymax=3600,
		ylabel={Empirical mean square bound},
		%ymajorgrids,
		axis background/.style={fill=white},
		legend style={legend cell align=left,align=left,draw=white!15!black}
		]
		\addplot[ybar,bar width=0.2,draw=black,fill=blue,area legend] plot table[row sep=crcr] {%
			0.8	3308.04099547188\\
			1.8	2318.99032282072\\
			2.8	1567.17764644847\\
			3.8	908.638495687647\\
			4.8	664.839935664505\\
			5.8	572.366991569963\\
			6.8	531.843299657987\\
			7.8	489.892681708289\\
			8.8	487.397233953426\\
		};
		\addlegendentry{present approach};
		
		\addplot[ybar,bar width=0.2,draw=black,fill=red,area legend] plot table[row sep=crcr] {%
			1.2	3496.94792131333\\
			2.2	2823.82066457343\\
			3.2	2063.54591157365\\
			4.2	1070.62743147457\\
			5.2	677.729555091094\\
			6.2	571.190987162819\\
			7.2	534.133998524972\\
			8.2	489.116573491372\\
			9.2	487.266043157536\\
		};
		\addlegendentry{Modified Controller of \cite{ref:Hokayem-12}};
		
		\addplot[ybar,bar width=0.2,draw=black,fill=green,area legend] plot table[row sep=crcr] {%
			1 0\\
			2 0\\
			3 0\\
			4	0\\
			5	0\\
			6	571.190987161151\\
			7	531.715074638705\\
			8	489.116573491372\\
			9	487.266043157536\\
		};
		\addlegendentry{Controller of \cite{ref:Hokayem-12}};
		
		\end{axis}
		\begin{axis}[%
		width=6.028in,
		height=4.754in,
		at={(0in,0in)},
		scale only axis,
		every outer x axis line/.append style={black},
		every x tick label/.append style={font=\color{black}},
		xmin=0,
		xmax=1,
		every outer y axis line/.append style={black},
		every y tick label/.append style={font=\color{black}},
		ymin=0,
		ymax=1,
		hide axis,
		axis x line*=bottom,
		axis y line*=left
		]
		\addplot [color=cyan,dashed,ultra thick,forget plot]
		table[row sep=crcr]{%
			0.725	0.135\\
			0.725	1\\
		};
		\addplot [color=cyan,dashed,ultra thick,forget plot]
		table[row sep=crcr]{%
			0.73	0.135\\
			0.73	1\\
		};
		\node[below right, align=left, text=blue, draw=black]
		at (rel axis cs:0.351,1) {\large{Controller of \cite{ref:Hokayem-12} is infeasible}};
		\end{axis}
		\end{tikzpicture}%
	\end{adjustbox}
	\caption{Empirical mean square bounds against $\authority$.}
	\label{Fig:normx}
\end{figure} 
\section{Epilogue}\label{s:conclusion}
We presented a tractable method for predictive control of linear stochastic control systems in the presence of incomplete and corrupt state information. We proved that given any fixed control bound, our receding horizon strategy yields a closed-loop system with mean-square bounded states whenever the system matrix $A$ is Lyapunov stable. Our results are valid when the process and measurement uncertainties are independent and Gaussian. In future work, we aim to incorporate control channel noise into our framework.

%%%%%%%%%%%%%%%%%%%%%%%%%%%%%%%%%%%%%%%%%%%%%%%%%%%%%%%%%%%%%%%%%%%%%%%%%%%%%%%%

%%%%%%%%%%%%%%%%%%%%%%%%%%%%%%%%%%%%%%%%%%%%%%%%%%%%%%%%%%%%%%%%%%%%%%%%%%%%%%%%
%\addtolength{\textheight}{-10cm}  

%%%%%%%%%%%%%%%%%%%%%%%%%%%%%%%%%%%%%%%%%%%%%%%%%%%%%%%%%%%%%%%%%%%%%%%%%%%%%%%%
\appendix
\section*{Appendix}

\begin{lemma}\cite[Lemma 8]{ref:Hokayem-12}
\label{lem:estimator error bound}
Consider the system \eqref{e:system} and assume that $ P_0 \succeq 0$. Then, there exists a time $ T^{\prime} \in \Nz $ and constant $ \rho > 0 $ such that $ \EE_{\YY_t} \left[ \norm{\st_t - \hat{\st}_t }^2 \right] = \trace(P_t) \leq \rho$ for all $t \geq T^{\prime}$. 
\end{lemma}

\begin{theorem}[{\cite[Theorem 1, Corollary 2]{ref:PemRos-99}}]
		\label{t:PemRos-99}
			Let \((X_t)_{t\in\Nz}\) be a family of real valued random variables on a probability space \((\Omega, \sigalg, \PP)\), adapted to a filtration \((\sigalg_t)_{t\in\Nz}\). Suppose that there exist scalars \(J, M, a > 0\) such that 
\begin{align}
				\EE_{\sigalg_t}[X_{t+1} - X_t] \le -a \quad\text{on the event } X_t > J,
			\text{ and} \\
				\EE\bigl[\abs{X_{t+1} - X_t}^4\,\big|\, X_0, \ldots, X_t\bigr] \le M\quad\text{for all }t\in\Nz.
\end{align}
			Then there exists a constant \(C > 0\) such that \(\sup_{t\in\Nz}\EE\bigl[ ((X_t)_+)^2 \mid \sigalg_0 \bigr] \le C\).
		\end{theorem}

%\newpage

\begin{proof}[Proof of Lemma \ref{lem:ortho stable}]
Let us consider the $\reachindex$-sub-sampled process of orthogonal subsystem of \eqref{e:decomposed st} by using \eqref{e:estimator st}:
$\hat{\st}_{\reachindex (t+1)}^o = \Aortho^{\reachindex} \hat{\st}_{\reachindex t}^o + \reachab_{\reachindex}(\Aortho,\Bortho)\control_{\reachindex t : \reachindex} + \Xi_{\reachindex t+ \reachindex - 1}^o . $ 
Define $z_{\reachindex t} \Let (\Aortho \transp)^{\reachindex t} \hat{\st}_{\reachindex t}^o $. It follows that
\begin{equation}
	\begin{aligned} 
				z_{\reachindex(t+1)} & = z_{\reachindex t} + (\Aortho\transp)^{\reachindex(t+1)} \reachab_\reachindex(\Aortho, \Bortho) \control_{\reachindex t:\reachindex} + (\Aortho\transp)^{\reachindex(t+1)}\Xi_{\reachindex t + \reachindex -1}^o \\ 
&=z_{\reachindex t} + \tilde{\control}_{\reachindex t} +  \tilde{\Xi}_{\reachindex t},
			\end{aligned}
			\end{equation}
			where, $ \tilde{\control}_{\reachindex t} = (\Aortho\transp)^{\reachindex(t+1)} \reachab_\reachindex(\Aortho, \Bortho) \control_{\reachindex t : \reachindex}, \tilde{\Xi}_{\reachindex t} = (\Aortho\transp)^{\reachindex(t+1)}\Xi_{\reachindex t + \reachindex - 1}^o.$
			Let \(\sigalg_{\reachindex t}\) denote the \(\sigma\)-algebra generated by \(\{\YY_{\reachindex \ell}\,|\, \ell = 0, 1, \ldots, t\}\). 
			Now, we have
			
	\[ \EE_{\sigalg_{\reachindex t}}\bigl[ z_{\reachindex (t+1)} - z_{\reachindex t}\bigr] = \EE_{\sigalg_{\reachindex t}} [\tilde{\control}_{\reachindex t} + \tilde{\Xi}_{\reachindex t} ] = \EE_{\sigalg_{\reachindex t}} [\tilde{\control}_{\reachindex t}] \]

%==============================================================================================

%==============================================================================================

			Let us consider the control sequence
			\begin{equation}\label{e:suggestedControl}
								\control_{\reachindex t:\reachindex} = - \reachab_{\reachindex}(\Aortho, \Bortho)^+ \Aortho^{\reachindex (t+1)} \sat_{r, \zeta}^\infty \bigl((\Aortho\transp)^{\reachindex t} \hat{\st}^o_{\reachindex t}\bigr).
			\end{equation}
It can be easily verified that $\control_{\reachindex t + \ell} \in \controlset$ for $\ell = 0, \ldots, \reachindex-1$ whenever $0 < \zeta < \frac{\authority}{\sqrt{d_o}\sigma_{1}\left(\reachab_{\reachindex}(\Aortho,\Bortho)^{+}\right)}$.  			
For the \(j\)-th component \(z^{(j)}_{\reachindex t}\) of \(z_{\reachindex t}\) we see that 
			\begin{align*}
				\EE_{\sigalg_{\reachindex t}}\bigl[ z^{(j)}_{\reachindex (t+1)} - z^{(j)}_{\reachindex t}\bigr] &= \EE_{\sigalg_{\reachindex t}}\bigl[ \tilde{\control}_{\reachindex t}^{(j)} \bigr] \\
				& = \left( (\Aortho \transp)^{\reachindex t} \EE \left[ (\Aortho \transp)^{\reachindex} \reachab_{\reachindex} (\Aortho, \Bortho) \control_{\reachindex t : \reachindex} \right] \right)^{(j)}	=-\zeta	, 		
			\end{align*}
			 whenever $z^{(j)}_{\reachindex t} = \left( (\Aortho \transp)^{\reachindex t} \hat{\st}_{\reachindex t}^o \right)^{(j)} > r$, and similarly
			\[
				\EE_{\sigalg_{\reachindex t}}\bigl[ z^{(j)}_{\reachindex (t+1)} - z^{(j)}_{\reachindex t} \bigr] = \zeta \quad\text{whenever }z^{(j)}_{\reachindex t} < -r.
			\]
If we define $X_t \Let z_{\reachindex t}^{(j)}, J \Let r, a \Let \zeta$, we can observe that the the first condition of Theorem \ref{t:PemRos-99} is satisfied for a given control sequence \eqref{e:suggestedControl}. For an arbitrary control sequence $\control_{\reachindex t : \reachindex}$ the first condition of Theorem \ref{t:PemRos-99} is satisfied if \eqref{e:drift1} is satisfied. Similarly, for $-X_t \Let z_{\reachindex t}^{(j)}$ the first condition of Theorem \ref{t:PemRos-99} is satisfied if \eqref{e:drift2} is satisfied. The rest of the proof follows as in the proof of \cite[Theorem 1.2]{ref:ChaRamHokLyg-12}. We include the salient steps for completeness.			
A straightforward computation relying on uniform boundedness of the control and the moment boundedness of $\Xi_{\reachindex t + \reachindex - 1}$ (see Proposition \ref{prop:bound on Xi}) shows that there exists an \(M > 0\) such that
			\begin{align*}
				& \EE\Bigl[\abs{z^{(j)}_{\reachindex (t+1)} - z^{(j)}_{\reachindex t}}^4\,\Big|\, z^{(j)}_{\reachindex t}, \ldots, z^{(j)}_0\Bigr] \\
				& = \EE\Bigl[ \abs{\tilde{\control}_{\reachindex t}^{(j)} + (\tilde{\Xi}_{\reachindex t})^{(j)}}^4 \, \Big|\, z^{(j)}_{\reachindex t}, \ldots, z^{(j)}_0 \Bigr] \leq M \quad\text{for all }t \geq \frac{T}{\reachindex},
			\end{align*}
		where $T \Let \reachindex \left \lceil \frac{T^{\prime}}{\reachindex} \right \rceil \geq T^{\prime}$ and $T^{\prime}$ is defined in Proposition \ref{prop:bound on Xi}. 
			Now \cite[Theorem 1]{ref:PemRos-99} guarantees the existence of constants \(C^{(j)}_1, C^{(j)}_2 > 0\), \(j = 1, \ldots, \dortho \), such that
			\[ \EE_{\YY_T}\Bigl[ \left( \bigl(z^{(j)}_{\reachindex t}\bigr)_+ \right)^2\Bigr] \leq C^{(j)}_1\quad\text{and}\quad \EE_{\YY_T}\Bigl[ \left( \bigl(z^{(j)}_{\reachindex t}\bigr)_-\right)^2\Bigr] \leq C^{(j)}_2 \text{ for all } t \geq \frac{T}{\reachindex}. \] 
			Since \(\abs{\delta} = \delta_+ + \delta_- = \delta_+ + (-\delta)_+\) for any \(\delta \in\R\), and for any $\delta \in \R^{\dortho}$, $ \norm{\delta}^2 = \sum_{j=1}^{\dortho} \abs{\delta^{(j)}}^2 \leq 2 \sum_{j=1}^{\dortho} \bigl((\delta^{(j)}_+)^2 + (\delta^{(j)}_-)^2\bigr),$ we see at once that the preceding bounds imply \[ \EE_{\YY_T} \bigl[\norm{z_{\reachindex t}}^2\bigr] < C\quad\text{for some constant \(C > 0\)}  \text{ for all } t \geq \frac{T}{\reachindex}. \]
			Since \(\hat{\st}_{\reachindex t}^o\) is derived from \(z_{\reachindex t}\) by an orthogonal transformation, it immediately follows that $ \EE_{\YY_T}\bigl[\norm{\hat{\st}^o_{\reachindex t}}^2\bigr] \le C$ for all $t \geq \frac{T}{\reachindex}$.  
			A standard argument (e.g., as in \cite{ref:ChaRamHokLyg-12}) now suffices to conclude from mean-square boundedness of the \(\reachindex\)-subsampled process \((\hat{\st}^o_{\reachindex t})_{t\in\Nz}\) the same property of the original process \((\hat{\st}^o_{t})_{t\in\Nz}\). In particular, there exists $\gamma_o$ such that
			\[ \EE_{\YY_T}\left[ \norm{\hat{\st}^o_t}^2\right] \leq \gamma_o \text{ for all } t \geq T.
 			\] 
\end{proof}
\begin{proof}[Proof of Theorem \ref{th:main}] Consider the objective function \eqref{e:cost}.
We substitute the stacked state vector \eqref{e:stacked state} in the objective function.
% and stacked control vector \eqref{e:stacked control} . 
\begin{align*}
& V_t =  \EE_{\YY_t} \left[ \sum_{k = 0}^{N-1} (\norm{\st_{t+k}}^2_{Q_k} + \norm{\control_{t+k}}^2_{R_k}  ) + \norm{\st_{t+N}}^2_{Q_N}\right] \\
& = \EE_{\YY_t} \left[ \norm{\calA\st_t + \calB\control_{t:N} + \calD \wnoise_{t:N}}^2_{\calQ} + \norm{\control_{t:N}}^2_{\calR} \right] \\
& = \EE_{\YY_t} \Bigl[ \norm{\calA\st_t}^2_{\calQ} + \norm{\calD\wnoise_{t:N}}^2_{\calQ} + \norm{\control_{t:N}}^2_{\alpha} + 2(\st_t \transp \calA \transp \calQ \calB + \wnoise_{t:N} \transp \calD \transp \calQ \calB  ) \control_{t:N} \\
& \quad + 2\st_t\transp\calA\transp\calQ\calD\wnoise_{t:N} \Bigr], \text{ where } \alpha = \calB\transp \calQ \calB + \calR. 
\end{align*}
Let $\beta_t \Let \EE_{\YY_t}\left[ \norm{\calA\st_t}^2_{\calQ} + \norm{\calD\wnoise_{t:N}}^2_{\calQ} + 2\st_t\transp\calA\transp\calQ\calD\wnoise_{t:N} \right] = \EE_{\YY_t}\left[ \norm{\calA\st_t}^2_{\calQ} + \norm{\calD\wnoise_{t:N}}^2_{\calQ}  \right]$, then 
\begin{equation} \label{e:objSolve}
V_t = \EE_{\YY_t} \Bigl[\norm{\control_{t:N}}^2_{\alpha} + 2(\st_t \transp \calA \transp \calQ \calB + \wnoise_{t:N} \transp \calD \transp \calQ \calB  ) \control_{t:N} \Bigr] + \beta_t .
\end{equation}
We substitute the stacked control vector \eqref{e:stacked control} in \eqref{e:objSolve} and for simplicity represent $\Delta_{\meas_t} = \meas_{t:N+1}-\hat{\meas}_{t:N+1} $ to get the following expression:
\begin{align*}
& V_t = \EE_{\YY_t} \Bigl[\norm{\offset_t + \gain_t \ee(\Delta_{\meas_t})}^2_{\alpha} + 2(\st_t \transp \calA \transp \calQ \calB + \wnoise_{t:N} \transp \calD \transp \calQ \calB  ) \bigl( \offset_t + \gain_t \ee(\Delta_{\meas_t}) \bigr) \Bigr] + \beta_t\\
& = \offset_t\transp \alpha \offset_t + \EE_{\YY_t} \Bigl[\norm{\gain_t \ee(\Delta_{\meas_t})}^2_{\alpha} + 2\offset_t \transp \alpha \gain_t \ee(\Delta_{\meas_t}) + 2(\st_t \transp \calA \transp \calQ \calB + \wnoise_{t:N} \transp \calD \transp \calQ \calB  ) \gain_t \ee(\Delta_{\meas_t}) \Bigr] \\
& + \EE_{\YY_t} \Bigl[ 2(\st_t \transp \calA \transp \calQ \calB + \wnoise_{t:N} \transp \calD \transp \calQ \calB  )\offset_t \Bigr] + \beta_t\\
& = \offset_t\transp \alpha \offset_t + \EE_{\YY_t} \Bigl[\norm{\gain_t \ee(\Delta_{\meas_t})}^2_{\alpha} + 2\offset_t \transp \alpha \gain_t \ee(\Delta_{\meas_t}) + 2(\st_t \transp \calA \transp \calQ \calB + \wnoise_{t:N} \transp \calD \transp \calQ \calB  ) \gain_t \ee(\Delta_{\meas_t}) \Bigr] \\
& \quad + 2\EE_{\YY_t} \Bigl[\st_t \transp \calA \transp \calQ \calB\offset_t \Bigr] + \beta_t .
\end{align*}
Since $\EE_{\YY_t}\left[ \st_t \right] = \hat{\st}_t$, we get the following expression:
\begin{equation}
\begin{aligned}\label{e:Vt}
& V_t = \offset_t\transp \alpha \offset_t + \EE_{\YY_t} \Bigl[\norm{\gain_t \ee(\Delta_{\meas_t})}^2_{\alpha} + 2(\offset_t \transp\alpha + \st_t \transp \calA \transp \calQ \calB + \wnoise_{t:N} \transp \calD \transp \calQ \calB  ) \gain_t \ee(\Delta_{\meas_t}) \Bigr] \\
& \quad + 2\hat{\st}_t \transp \calA \transp \calQ \calB\offset_t + \beta_t .
\end{aligned}
\end{equation}
Let us consider the term $\EE_{\YY_t} \left[ \offset_t \transp \alpha \gain_t \ee(\Delta_{\meas_t})   \right]$. By observing $\EE_{\YY_t} \left[ \ee_i( \meas_{t+i} - \hat{\meas}_{t+i}) \right] = \zeros$ for each $i = 1, \ldots, N+1$, we get
\begin{align}\label{e:etay}
\EE_{\YY_t} \left[ \offset_t \transp \alpha\gain_t \ee(\Delta_{\meas_t})   \right] &= \offset_t \transp \alpha \bmat{\theta_{0,t} \\ \theta_{1,t}\\ \vdots \\ \theta_{N-1,t}}\ee_0(\meas_t - \hat{\meas}_t) \notag \\
&=\offset_t \transp \alpha \gain_t^{(:,t)}\ee_0(\meas_t - \hat{\meas}_t),
\end{align}
where $\gain_t^{(:,t)} \Let \bmat{\theta_{0,t} \\ \theta_{1,t}\\ \vdots \\ \theta_{N-1,t}} $ represents the first $q$ columns of the gain matrix $\gain_t$. 
%Then we have 
%\begin{equation}
%\EE_{\YY_t} \left[ \offset_t \transp \gain_t \ee(\Delta_{\meas_t})   \right] = \offset_t \transp \gain_t^{(:,t)}\ee_0(\meas_t - \hat{\meas}_t).
%\end{equation}
Let us consider the term $\EE_{\YY_t} \Bigl[\norm{\gain_t \ee(\Delta_{\meas_t})}^2_{\alpha} \Bigr]$. In order to make the offline computation easy we do the following manipulation:
\begin{align}
& \EE_{\YY_t} \Bigl[\norm{\gain_t \ee(\Delta_{\meas_t})}^2_{\alpha} \Bigr] \nonumber \\
& = \EE_{\YY_t} \Biggl[\norm{\bmat{\gain_t^{(:,t)} & \gain_t^{\prime}} \bmat{\ee_0(\meas_t - \hat{\meas}_t) \\ \ee^{\prime}(\meas_{t+1:N} - \hat{\meas}_{t+1:N})} }^2_{\alpha} \Biggr] \nonumber \\
& = \EE_{\YY_t} \Biggl[\bmat{\ee_0(\meas_t - \hat{\meas}_t) \\ \ee^{\prime}(\meas_{t+1:N} - \hat{\meas}_{t+1:N})} \transp \bmat{\gain_t^{(:,t)} & \gain_t^{\prime}}\transp  {\alpha} \bmat{\gain_t^{(:,t)} & \gain_t^{\prime}} \bmat{\ee_0(\meas_t - \hat{\meas}_t) \\ \ee^{\prime}(\meas_{t+1:N} - \hat{\meas}_{t+1:N})} \Biggr] \nonumber \\
& = \trace \Biggl( \alpha \gain_t^{(:,t)} \ee_0(\meas_t - \hat{\meas}_t) \ee_0(\meas_t - \hat{\meas}_t) \transp (\gain_t^{(:,t)})\transp  \nonumber \\
& \quad + \alpha \gain_t^{\prime} \EE_{\YY_t}\Bigl[\ee^{\prime}(\meas_{t+1:N} - \hat{\meas}_{t+1:N})\ee^{\prime}(\meas_{t+1:N} - \hat{\meas}_{t+1:N})\transp (\gain_t^{\prime})\transp \Bigr]   \Biggr) \nonumber \\
& = \trace(\alpha \gain_t^{(:,t)} \Pi_{\meas_t} (\gain_t^{(:,t)})\transp) + \trace (\alpha \gain_t^{\prime} \Sigma_{\ee} (\gain_t^{\prime})\transp) \label{e:Sigmay}, 
\end{align}
where $\Pi_{\meas_t} = \ee_0(\meas_t - \hat{\meas}_t) \ee_0(\meas_t - \hat{\meas}_t) \transp $ and $\Sigma_{\ee} = \EE\Bigl[\ee^{\prime}(\meas_{t+1:N} - \hat{\meas}_{t+1:N})\ee^{\prime}(\meas_{t+1:N} - \hat{\meas}_{t+1:N})\transp \Bigr]$.
Further, we simplify the term $\EE_{\YY_t} \Bigl[ \wnoise_{t:N} \transp \calD \transp \calQ \calB \gain_t \ee(\Delta_{\meas_t})\Bigr]$ as follows:
\begin{align}
& \EE_{\YY_t} \Bigl[\wnoise_{t:N} \transp \calD \transp \calQ \calB \gain_t \ee(\Delta_{\meas_t})\Bigr] \nonumber \\
&= \EE_{\YY_t} \Biggl[ \wnoise_{t:N} \transp \calD \transp \calQ \calB  \bmat{\gain_t^{(:,t)} & \gain_t^{\prime}} \bmat{\ee_0(\meas_t - \hat{\meas}_t) \\ \ee^{\prime}(\meas_{t+1:N} - \hat{\meas}_{t+1:N})}\Biggr] \nonumber \\
&= \trace \left( \calD \transp \calQ \calB\gain_t^{\prime} \EE_{\YY_t} \Biggl[\ee^{\prime}(\meas_{t+1:N} - \hat{\meas}_{t+1:N})\wnoise_{t:N} \transp \Biggr] \right) \nonumber \\
&= \trace \left( \calD \transp \calQ \calB\gain_t^{\prime} \Sigma_{\ee^{\prime}\wnoise} \right) \label{e:SigmaPsi},
\end{align}
where $\Sigma_{\ee^{\prime}\wnoise} = \EE \Biggl[\ee^{\prime}(\meas_{t+1:N} - \hat{\meas}_{t+1:N})\wnoise_{t:N} \transp \Biggr]$.
We simplify the term $\EE_{\YY_t} \Bigl[\st_t \transp \calA \transp \calQ \calB \gain_t \ee(\Delta_{\meas_t})\Bigr]$ as follows:
\begin{align}
& \EE_{\YY_t} \Bigl[\st_t \transp \calA \transp \calQ \calB \gain_t \ee(\Delta_{\meas_t})\Bigr] \notag \\
& = \EE_{\YY_t} \Bigl[(\st_t - \hat{\st}_t)\transp \calA \transp \calQ \calB \gain_t \ee(\Delta_{\meas_t})\Bigr] + \EE_{\YY_t} \Bigl[ \hat{\st}_t\transp \calA \transp \calQ \calB \gain_t \ee(\Delta_{\meas_t})\Bigr] \notag \\
& = \EE_{\YY_t} \Bigl[e_t \transp \calA \transp \calQ \calB \gain_t \ee(\Delta_{\meas_t})\Bigr] + \hat{\st}_t\transp \EE_{\YY_t} \Bigl[  \calA \transp \calQ \calB \gain_t \ee(\Delta_{\meas_t})\Bigr] \notag \\
& = \EE_{\YY_t} \Bigl[e_t \transp \calA \transp \calQ \calB \gain_t \ee(\Delta_{\meas_t})\Bigr] + \hat{\st}_t\transp \calA \transp \calQ \calB \gain_t^{(:,t)} \ee_0(\meas_t - \hat{\meas}_t) \notag \\
& =  \trace\Bigl(\calA \transp \calQ \calB \gain_t^{\prime} \EE_{\YY_t} \Bigl[ \ee^{\prime} (\meas_{t+1:N} - \hat{\meas}_{t+1:N})e_t \transp \Bigr] \Bigr)  + \hat{\st}_t\transp \calA \transp \calQ \calB \gain_t^{(:,t)} \ee_0(\meas_t - \hat{\meas}_t) \notag \\
& =  \trace\Bigl(\calA \transp \calQ \calB \gain_t^{\prime} \Sigma_{e\ee^{\prime}} \Bigr)  + \hat{\st}_t\transp \calA \transp \calQ \calB \gain_t^{(:,t)} \ee_0(\meas_t - \hat{\meas}_t), \label{e:Sigmae}
\end{align}
where $\Sigma_{e\ee^{\prime}} = \EE_{\YY_t} \Bigl[ \ee^{\prime} (\meas_{t+1:N} - \hat{\meas}_{t+1:N})e_t \transp \Bigr]$.
We substitute \eqref{e:etay}, \eqref{e:Sigmay}, \eqref{e:SigmaPsi}, \eqref{e:Sigmae} in \eqref{e:Vt}, and ignore the terms independeent of the decision variables. We get the desired objective function \eqref{e:obj}:
\begin{equation}
\begin{aligned}
& \offset_t\transp \alpha \offset_t + \trace(\alpha \gain_t^{(:,t)} \Pi_{\meas_t} (\gain_t^{(:,t)})\transp) + \trace (\alpha \gain_t^{\prime} \Sigma_{\ee} (\gain_t^{\prime})\transp) +2\hat{\st}_t \transp \calA \transp \calQ \calB\offset_t \\
& +  2(\offset_t \transp\alpha + \hat{\st}_t \transp \calA \transp \calQ \calB ) \gain_t^{(:,t)}\ee_0(\meas_t - \hat{\meas}_t) + 2\trace\left(\calD \transp \calQ \calB\gain_t^{\prime} \Sigma_{\ee^{\prime}\wnoise}\right) + 2\trace \Bigl( \calA\transp \calQ \calB \gain_t^{\prime} \Sigma_{e \ee^{\prime}} \Bigr)  .
\end{aligned}
\end{equation}
%By observing that $\EE_{\YY_t}\left[ \ee(\meas_{t:N+1} - \hat{\meas}_{t:N+1}) \right] = \zeros$ and $\EE_{\YY_t} \left[ \st_t \right] = \hat{\st}_t$, and removing terms independent of decision variables, we get the desired objective function \eqref{e:obj}. 
Therefore, the objective function in \eqref{e:cost} is equivalent to \eqref{e:obj} under the constraints \eqref{e:stacked state} and \eqref{e:stacked control}. Since the expected value of a convex function is convex and \eqref{e:obj} is obtained by substituting the affine functions of the decision variables into a quadratic function, the objective function \eqref{e:obj} is convex quadratic \cite{boyd_convex_optimization}.
The constraint \eqref{e:hard bound on control} is an affine function of the decision variables $\offset_t,\gain_t$; hence it is convex. The constraint \eqref{e:hard bound on control} is equivalent to the hard control constraint \eqref{e:controlset}. This constraint is obtained by utilizing the fact that $\ee: \R^{q(N+1)}  \rightarrow \R^{q(N+1)}$ is component-wise symmetric about the origin and the innovation sequence is Gaussian.
The stability constraints \eqref{e:stability constraint} are obviously convex. To see their feasibility, let us consider $\control_{t:\reachindex}$ \eqref{e:kappa blocks control}, set $(\offset_t)_{1:\reachindex m} = - \reachab_{\reachindex}(\Aortho, \Bortho)^+ \Aortho^{\reachindex} \sat_{r, \zeta}^\infty \bigl(\hat{\st}^o_{t}\bigr), (\gain_t)_{1:\reachindex m} = \zeros$,  and observe that $\control_{t:\reachindex}$ satisfies \eqref{e:stability constraint} for all $t$. Moreover, for $\ell = 0, \ldots, \reachindex-1$, $\norm{\control_{t + \ell}}_{\infty} \leq \norm{\control_{t + \ell}} \leq \norm{\control_{t : \reachindex}} \leq \sigma_1(\reachab_{\reachindex}(\Aortho, \Bortho)^+)\sqrt{d_o}\zeta \leq \authority$ whenever $0 < \zeta < \frac{\authority}{\sqrt{d_o}\sigma_{1}\left(\reachab_{\reachindex}(\Aortho,\Bortho)^{+}\right)}$, which implies $\control_{t+\ell} \in \controlset$ \cite[Theorem 4]{ref:PDQ-15}.
\end{proof}
\begin{proof}[Proof of Theorem \ref{claim:stbl}] 
Choose $T \Let N_r \left \lceil \frac{T^{\prime}}{N_r} \right \rceil \geq T^{\prime}$, where $T^{\prime}$ is according to the Lemma \ref{lem:estimator error bound}, then for every $t \geq T$  
\begin{align*}
\EE_{\YY_T}\left[ \norm{\st_t}^2 \right] &\leq 2\EE_{\YY_T}\left[ \norm{\st_t - \hat{\st}_t}^2\right] + 2\EE_{\YY_T}\left[ \norm{\hat{\st}_t}^2\right] \\
&\leq 2\rho + 2\EE_{\YY_T}\left[ \norm{\hat{\st}_t}^2\right] \text{ from \cite[Lemma 8]{ref:Hokayem-12}}\\
&= 2\rho + 2\left( \EE_{\YY_T}\left[ \norm{\hat{\st}^s_t}^2\right] + \EE_{\YY_T}\left[ \norm{\hat{\st}^o_t}^2\right] \right)\\
& \leq 2 \left( \rho + \gamma_s + \EE_{\YY_T}\left[ \norm{\hat{\st}^o_t}^2\right] \right) \text{ from \cite[Lemma 10]{ref:Hokayem-12}}\\
& \leq 2 \left( \rho + \gamma_s + \gamma_o \right) \text{ from Lemma \ref{lem:ortho stable}}\\
& \teL \gamma^{\prime} \text{ for all } t \geq T.
\end{align*}
By using the tower property of the conditional expectation, we get 
\begin{equation*}
\EE_{\YY_0} \left[ \EE_{\YY_T}\left[ \norm{\st_t}^2 \right] \right] = \EE_{\YY_0}\left[ \norm{\st_t}^2 \right] \leq \gamma^{\prime} \text{ for all } t \geq T.
\end{equation*} 
For $t = 0, \cdots, T-1$, $\st_t = \A^{t-1}\st_0 + \reachab_t(A,B)\bmat{\control_0\transp & \cdots & \control_{t-1}\transp}\transp + \reachab_t(\A,I_d)\bmat{\wnoise_0\transp & \cdots & \wnoise_{t-1}\transp}\transp$. Because $\wnoise_i$'s are mean zero Gaussian and independent of $\st_0$, and the controls are bounded, we have
\begin{equation*}
\EE_{\YY_0}\left[ \norm{\st_t}^2 \right] \leq \gamma_t \quad \text{ for } t = 0, \cdots, T-1.
\end{equation*}
Define $\gamma \Let \max \left\{ \gamma^{\prime}, \gamma_t \mid t = 0, \cdots,T-1 \right \}$, we conclude that 
$ \sup_{t\in \Nz} \EE_{\YY_0}\left[ \norm{\st_t}^2 \right] \leq \gamma.$   
\end{proof}

%%%%%%%%%%%%%%%%%%%%%%%%%%%%%%%%%%%%%%%%%%%%%%%%%%%%%%%%%%%%%%%%%%%%%%%%%%%%%%%%

\end{document}